\title{Restricted integer partition functions
}
\author{Noga Alon
\thanks{Sackler School of Mathematics
and Blavatnik School of Computer Science,
Tel Aviv University,
Tel Aviv 69978, Israel.
Email: {\tt nogaa@tau.ac.il}.
Research supported in part by an ERC Advanced
grant, by a USA-Israeli BSF grant and by 
the Israeli I-Core program.}
}

\documentclass[11pt]{article}
\usepackage{amsfonts,amssymb,amsmath,latexsym}
\newenvironment{proof}
      {\medskip\noindent{\bf Proof.}\hspace{1mm}}
      {\hfill$\Box$\medskip}

\def\qed{\ifvmode\mbox{ }\else\unskip\fi\hskip 1em plus 10fill$\Box$}

\newtheorem{theo}{Theorem}[section]

\newtheorem{lemma}[theo]{Lemma}
\newtheorem{coro}[theo]{Corollary}

\newcommand{\CC}{{\cal C}}

\begin{document}
\date{}

\maketitle

\begin{abstract}
For two sets $A$ and $M$ of positive integers and for a positive
integer $n$, let $p(n,A,M)$ denote the number of partitions of $n$
with parts in $A$ and multiplicities in $M$, that is, the number of
representations of $n$ in the form $n=\sum_{a \in A} m_a a$
where $m_a \in M \cup \{0\}$ for all $a$, and all numbers $m_a$ but
finitely many are $0$. It is shown that there are infinite sets $A$ and
$M$ 
so that $p(n,A,M)=1$ for every positive integer $n$. This settles
(in a strong form) a problem of Canfield and Wilf. It is also shown
that there is an infinite set $M$ and constants $c$ and $n_0$ so
that for $A=\{k!\}_{k \geq 1}$ or for $A=\{k^k\}_{k \geq 1}$,
$0<p(n,A,M) \leq n^c$ for all $n>n_0$. This answers a question of
Ljuji\'c and Nathanson.
\end{abstract}

\section{Introduction}
For two sets $A$ and $M$ of positive integers and for a positive
integer $n$, let $p(n,A,M)$ denote the number of representations of
$n$ in the form
\begin{equation}
\label{e11}
n=\sum_{a \in A} m_a a
\end{equation}
where $m_a \in M \cup \{0\}$ for all $a$, and all numbers $m_a$ but
finitely many are $0$.  We say that the function $p=p(n,A,M)$ has
polynomial growth  if there exists an absolute constant $c$ and an
integer $n_0$ so that $p(n,A,M) \leq n^c$ for all $n>n_0$.
Canfield and Wilf \cite{CW} raised the following question.
\vspace{0.3cm}

\noindent
{\bf Question 1 [Canfield and Wilf \cite{CW}]}\,

Do there exist two
{\bf infinite} sets $A$ and $M$ so that $p(n,A,M)>0$ for all
sufficiently large $n$ and yet $p$ has polynomial growth ?
\vspace{0.3cm}

\noindent
Ljuji\'c and Nathanson \cite{LN} observed, among other things, that
this cannot be the case if the set $A$ has at least $\delta \log
n$ members in $[n]=\{1,2,\ldots ,n\}$ for all sufficiently large $n$, 
where  $\delta>0$ is any positive constant, and asked the following
two more specific questions.
\newpage

\noindent
{\bf Question 2 [Ljuji\'c and Nathanson \cite{LN}]}\,

Let $A=\{k!\}_{k=1}^{\infty}$. Does there exist an infinite set
$M$ of positive integers so that
$p(n,A,M)>0$ for all
sufficiently large $n$, and  $p$ has polynomial growth ?
\vspace{0.3cm}

\noindent
{\bf Question 3 [Ljuji\'c and Nathanson \cite{LN}]}\,

Let $A=\{k^k\}_{k=1}^{\infty}$. Does there exist an infinite set
$M$ of positive integers so that
$p(n,A,M)>0$ for all
sufficiently large $n$ and $p$ has polynomial growth ?
\vspace{0.3cm}

\noindent
In this note we prove that the answer to all three questions above is
positive. 

Our first result is simple, and shows that the answer to the
first question is positive in the strongest possible way: there are
infinite sets $A$ and $M$ so that $p(n,A,M)=1$ for all $n$.
This is stated in the following Theorem.
\begin{theo}
\label{t11}
There are two infinite sequences of positive integers $A$ and $M$
so that $p(n,A,M)=1$ for every positive integer $n$.
\end{theo}
The proof is by an explicit construction, which is general enough
to give examples with sets $A$ that grow at any desired rate which
is faster than exponential. This rate of growth is tight, by the
observation of Ljuji\'c and Nathanson mentioned above.

We also prove the following result, which settles Questions 2 and
3.
\begin{theo}
\label{t12}
Let $A$ be an infinite set of positive integers, and suppose that
$1 \in A$ and 
$$
|A \cap [n]| =(1+o(1)) \frac{\log n}{\log \log n},
$$ 
where the
$o(1)$ term tends to $0$ as $n$ tends to infinity. Then there
exists $n_0$ and an infinite set $M$ of 
positive integers so that
$0<p(n,A,M)<n^{8+o(1)}$ for all
$n>n_0$.
\end{theo}
The upper estimate $n^{8+o(1)}$ can be improved
and we make no attempt to optimize
it here. The proof of this theorem is probabilistic and
can be easily modified to provide the existence
of such sets $M$ for many other sparse infinite sets $A$.


The rest of this note is organized as follows. In Section 2 we
present the simple proof of Theorem \ref{t11}. In Section 3
we describe the probabilistic construction of the set $M$
used in the proof of Theorem \ref{t12}. The proof it
satisfies the required properties and hence establishes Theorem
\ref{t12} is described in Section 4. Section 5 contains some
concluding remarks.

All logarithms throughout the paper are in base $2$, 
unless otherwise specified. 

\section{The proof of Theorem \ref{t11}}

\begin{proof}
Let $N \cup \{0\}=B \cup C$, $B \cap C=\emptyset$ be a partition 
of the set of all non-negative integers into two disjoint infinite
sets $B$ and $C$.  Let 
$$
D=\{\sum_{i \in B'} 2^i~:~ B' \subset B\}
$$
be the set of all sums of powers of two in which each exponent lies
in B. Similarly, put
$$
E=\{\sum_{j \in C'} 2^j~:~ C' \subset C\}.
$$
This is the set of 
all sums of powers of two in which each exponent is a member of
$C$.
Note that both $D$ and $E$ contain $0$, as the sets $B'$ and
$C'$ in their definition may be taken to be empty.
Since every non-negative integer has a unique binary
representation, it
is clear that each such integer has a unique representation as a sum
of an element of $D$ and an element of $E$.

Define, now,
$$
A=\{2^d~:~  d \in D\},
$$
and
$$
M=\{\sum_{e \in E'} 2^e~:~ E' \subset E \}.
$$
Clearly both  $A$ and $M$ are infinite.
We claim that every positive integer has a
unique representation of the form (\ref{e11}) with these $A$ and
$M$, that is, $p(n,A,M)=1$ for all $n \geq 1$.
Indeed, a general expression of the form 
$\sum_{a \in A} m_a a$ with $m_a \in M$ for all $a$ satisfies
$$
\sum_{a \in A} m_a a=\sum_{d \in D} \sum_{e \in E'_d} 2^{e+d}.
$$
Let $n=2^{t_1}+2^{t_2} + \ldots + 2^{t_r}$ be the binary
representation of $n$, with $0 \leq t_1<t_2 < \ldots < t_r$. By the 
fact that each nonnegative integer has a unique representation as a
sum of an element of D and an element of E, there are, for each
$1 \leq j \leq r$, unique $d_i\in D$ and $e_i \in E$ so that
$t_i=d_i+e_i$. The elements $d_i$ are not necessarily distinct.
Let $D'=\{d_1, d_2 \ldots ,d_r\}$ be the set of all distinct ones.
For each
$d \in D'$ define $m_{2^d}=\sum_{i: d_i=d} 2^{e_i}$
and observe that 
$$
n=\sum_{d \in D} m_{2^d} \cdot 2^d.
$$
Therefore $n$ has a representation of the form (\ref{e11}). It is
not difficult to check that this representation is unique, that is,
the elements $d$ and $m_{2^d}$ in the expression above can be
reconstructed in the unique way described above from the binary
representation   of $n$. This shows that indeed $p(n,A,M)=1$ for
all $n$, completing the proof of Theorem \ref{t11}.  
\end{proof}

Note that by choosing the set $B$ so that $|B \cap
\{0,1,\ldots,r\}|=r-g(r)$, where $g(r)$ is an arbitrary monotone
function  of $r$ that tends to infinity, and $r-g(r)$ also tends to
infinity in a monotone way, we get that for every integer $r \geq
1$:
$$
|A \cap [2^{2^r}]|=|D \cap [2^{r}]| 
=|B \cap \{r\}|+2^{|B \cap \{0,1,\ldots ,r-1\}|}
=2^{r-1-g(r-1)} +O(1).
$$
As $r = \log \log n$ this shows that $|A \cap [n]|$ is 
$$
\Theta(\frac{\log n}{2^{g(\log \log n)}})+O(1),
$$ showing that by an appropriate choice of $g$ we can get sets $A$
whose growth function is at least $\log n/w(n)$ for an arbitrary
slowly growing function $w(n).$

\section{The probabilistic construction}

The set $M$ defined for the proof of Theorem \ref{t12} is a
union of the form $M=\cup_{a \in A} M_a$. Each set $M_a$
is a union $M_a=\cup_{i: 2^i \geq a} M_{a,i}$. Each of the sets 
$M_{a,i}$ is a random subset of $[2^i]$ obtained by picking each
number in $[2^i]$, randomly and independently, to be a member of 
$M_{a,i}$ with probability $\frac{i^6}{2^i}$ (if this ratio exceeds
$1$ then all members of $[2^i]$ lie in $M_{a,i}$, clearly this
happens only for finitely many values of $i$.)

We claim that $M$ satisfies the required properties with high 
probability (that is, with probability that tends to $1$ as $n_0$
tends to infinity). This is proved in the next section. The fact
that $p=p(n,A,M)$ has polynomial growth is simple: we show that
with high probability the set $M$ is sparse enough to ensure that
the total number of expressions of the form (\ref{e11}) that are at
most $n$ does not exceed $n^{8+o(1)}$. The fact that with high probability
$p(n,A,M)>0$ for all sufficiently large $n$ is more complicated and
requires some work. It turns out to be convenient to
restrict attention to
expressions (\ref{e11}) of a special form that enable us to control
their behaviour in an effective manner. We can then apply the
Janson Inequality (c.f., \cite{AS}, Chapter 8) to derive the
required result.

\section{The proof of Theorem \ref{t12}}

\subsection{Polynomial growth}

Let $A$ be as in Theorem \ref{t12}, and let 
$M=\cup_{a \in A} \cup_{i: 2^i \geq a} M_{a,i}$ be  as in the
previous section. In this subsection we show that with high
probability the function $p(n,A,M)$ has polynomial growth.
\begin{lemma}
\label{l31}
For every $\epsilon >0$ there exists an $n_0=n_0(\epsilon)$ so that
with probability at least $1-\epsilon$, $M$ is infinite
and $|M \cap [n]| < \log^8
n-1$ for all $n>n_0$.
\end{lemma}
\begin{proof}
It is obvious that $M$ is infinite with probability $1$. We proceed
to prove the main part of the lemma.

Let $m$ be a sufficiently large integer. If $2^i<m$ then
$m$ cannot lie in $M_{a,i}$. If $2^i \geq m$, then the probability
that $m \in M_{a,i}$ is $\frac{i^6}{2^i}.$ Since for large $i$,
$\frac{(i+1)^6}{2^{i+1}} < \frac{2}{3} \frac{i^6}{2^i}$ it follows
that for $a \leq m$ the probability 
$Pr[m \in M_a=\cup_{i: 2^i \geq m} M_{a,i}]$ 
is smaller than
$3 \frac{\log^6 m}{m}.$ Similarly, for $a >m$, 
$Pr[ m \in M_a] < 3 \frac{\log^6 a}{a}$.  Summing over all
$a \in A,~ a >m$ and using the fact that $A$ is sparse we conclude that
$Pr[ m \in \cup_{a >m} M_a] < 10 \frac{\log^6 m}{m}$. Since there are
$o(\log m)$ members of $A$ that are smaller than $m$ we also get that
(since $m$ is large) $Pr[m \in \cup_{a\leq m}M_a ] < 0.9
\frac{\log^7 m}{m}$. Altogether, the probability that $m$ lies in
$M=\cup_a M_a$ does not exceed $\frac{\log^7 m}{m}$.

It follows that the expected value of $|M \cap [n]|$ is smaller
than
$O(1)+ \frac{1}{8}\log_e ^8 n< 0.5 \log_2^8 n$ (where the $O(1)$ 
term is added to account for
the small values of $m$). As this cardinality is a sum of
independent random variables we can apply the
Chernoff-Hoeffding Inequality
(c.f., e.g., \cite{AS}, Appendix A) and conclude that the
probability that $|M \cap [n]|$ is at least $\log^8 n -1$ is at
most $e^{-\Omega(\log^8 n)}$. For sufficiently large $n_0$
the sum $\sum_{n \geq n_0} e^{-\Omega(\log^8 n)}< \epsilon$,
completing the proof of the lemma.
\end{proof}

\begin{coro}
\label{c32}
Let $A$ and $M$ be as above.
For any $\epsilon>0$ there is an  $n_0=n_0(\epsilon)$ so that
with probability at least $1-\epsilon$, $M$ is infinite and 
$\sum_{n \leq m} p(n,A,M)
\leq m^{8+o(1)}$ for all $m>n_0$.
\end{coro}
\begin{proof}
Let $\epsilon$ and $n_0=n_0(\epsilon)$ be as in Lemma \ref{l31} and
suppose that $M$ satisfies the assertion of the lemma (this
happens with probability at least $1-\epsilon$). Then, the number
of representations of the form (\ref{e11}) of integers $n$ that do
not exceed $m$ is at most $|(M \cap [m]) \cup \{0\}|^{|A \cap
[m]|}$, as all coefficients $m_a$ and all numbers $a$ with a
nonzero coefficient must be at most $m$. This expression
is at most $(\log^8 m)^{(1+o(1)) \log m / \log \log m}
=m^{8+o(1)}$, as needed.
\end{proof}

\subsection{Representing all large integers}

In this subsection we prove that with high probability every
sufficiently large number $n$ has a representation of the form
(\ref{e11}) where $A$ and $M$ are as above. To do so, it is
convenient to insist on a representation of a special form,
which we proceed to define. Put $A'=\{a \in A: a \leq n^{1/3}\}.$
Let $q=|A'|$, $A'=\{a_1, a_2 , \ldots ,a_q\}$ where
$1=a_1<a_2 <\ldots <a_q$. Thus $q =(\frac{1}{3}+o(1)) \frac{\log
n}{\log \log n}$. Let $i_1$ be the smallest integer $t$ so that
$2^{t} \geq n$. For $2 \leq j \leq q$, let $i_j$ be the smallest
integer $t$ so that $2^t \geq \frac{n}{a_j \log n}$. Therefore,
$n \leq 2^{i_1} < 2n$ and $\frac{n}{a_j \log n} \leq 2^{i_j}
<\frac{2n}{a_j \log n}$ for all $2 \leq j \leq q$. 
Since $a_j \leq n^{1/3}$ for all $a_j \in A'$, it follows that
$i_j > \frac{1}{2} \log n$ for all admissible $j$.

We say that
$n$ has a special representation as a partition with parts in $A$
and
multiplicities in $M$ (for short: n has a special representation)
if there is a representation of the form (\ref{e11}) where
$m_j \in M_{a_j,i_j}$ for all $1 \leq j \leq q$.
\begin{lemma}
\label{l33}
For all sufficiently large $n$, the probability that $n$ does not
have a special representation is at most
$e^{-\Omega(\log^5 n)}.$ Therefore, for any $\epsilon>0$ there is
an $n_0=n_0(\epsilon)$ such that the probability that there is 
a special representation for every integer $n>n_0$ is at least
$1-\epsilon$.
\end{lemma}

The assertion of Theorem \ref{t12} follows from Corollary \ref{c32}
(with $\epsilon <1/2$) and Lemma \ref{l33} (with $\epsilon<1/2$)
that supply the existence of an infinite set $M$ satisfying the
conclusions of the theorem.

In the proof of Lemma \ref{l33} we apply the Janson 
Inequality (c.f. \cite{AS}, Chapter 8). We first state it 
to set the required notation. Let $X$ be a finite set, and let $R
\subset X$ be a random subset of $X$ obtained by picking each
element $r \in X$ to be a member of $R$, randomly and
independently, with probability $p_r$. Let $\CC=\{ C_i \}_{i \in I}$ be
a collection of subsets of $X$, let $B_i$ denote the event that
$C_i \subset R$, and let $i \sim j$ denote the fact that 
$i \neq j $ and $C_i \cap C_j \neq \emptyset$.
Let $\mu=\sum_{i \in I} Pr[B_i]=\sum_{i \in I} \prod_{j \in C_i}
p_j$ be the expected number of events $B_i$ that hold, and define
$\Delta=\sum_{i,j \in I, i\sim j} Pr[B_i \cap B_j]$ where the sum is
computed over ordered pairs. The inequality we need is the
following.
\begin{lemma}[The Janson Inequality]
\label{l34}
In the notation above, if $\Delta \leq D$ with $D \geq \mu$
then the probability that no event $B_i$ holds is at most
$e^{-\mu^2/2 D}.$
\end{lemma}
Note that the above statement is an immediate consequence of the
two Janson inequalities described in \cite{AS}, Chapter 8. If 
$\Delta < \mu$ than the above statement follows from the first
inequality, whereas if $\Delta \geq \mu$ then it follows from the
second.

We can now prove Lemma \ref{l33}.

\begin{proof}
We apply the Janson inequality as stated in Lemma \ref{l34} above.
The set $X$ here is a disjoint union of the sets
$X_j=[2^{i_j}]$, $1 \leq j \leq q$, and each element of $X_j$ is chosen
with probability $\frac{i_j^6}{2^{i_j}}$. The sets in the
collection of sets $\CC$ are all sequences of the form $(m_1,m_2,
\ldots, m_q)$ with $m_j \in X_j$ so that 
\begin{equation}
\label{e31}
\sum_{j=1}^q m_j a_j=n.
\end{equation}
Note that there are exactly $\prod_{j=2}^q 2^{i_j}$ such sets,
as for any choice of $m_j \in X_j$, $2 \leq j \leq q$
there is a unique choice of
$m_1 \in X_1$ so that (\ref{e31}) holds. Therefore, in the notation
above
$$
\mu=\prod_{j=2}^q 2^{i_j} \prod_{j=1}^q \frac{i_j^6}{2^{i_j}}=
\frac{1}{2^{i_1}} \prod_{j=2}^{q} i_j^6= n^{1-o(1)},
$$
where the last equality follows from the fact that
$\frac{1}{2} \log n \leq i_j \leq \log n$ for all $2 \leq j \leq q$
and the fact that $q=(\frac{1}{3}+o(1)) \frac{\log n}{\log \log
n}.$

We proceed with the estimation of the quantity $\Delta$ that
appears in the inequality. This is the sum, over all ordered pairs of
sequences
$m_j^{(1)}$ and $m_j^{(2)}$ with $m_j^{(1)},m_j^{(2)} \in X_j$,
where $\sum m_j^{(1)} a_j = \sum m_j^{(2)} a_j=n$ and
for at least one $r$, $m_r^{(1)}=m_r^{(2)}$, of the probability that
both $m_j^{(1)}$ and $m_j^{(2)}$ belong to $M_{a_j,i_j}$ for all
$j$.

Write $\Delta=\sum_{\ell=1}^q \Delta_{\ell}$, where $\Delta_{\ell}$
is the sum of these probabilities over all pairs 
$m_j^{(1)}$ and $m_j^{(2)}$ as above for which
$\ell=\min\{j: m_j^{(1)} \neq m_j^{(2)}~\}.$
\vspace{0.3cm}

\noindent
{\bf Claim 1:}\, 
$$
\Delta_1 \leq \frac{1}{2^{i_1}} \prod_{j=1}^q i_j^6
\sum_{\emptyset \neq I \subset \{2,3,\ldots,q\}}
\frac{i_1^6}{2^{i_1}} \prod_{j>1, j \not \in I} i_j^6.
$$

Indeed, for each choice of the sequence $m_j^{(1)}$ the
contribution to $\Delta_1$  arises from sequences $m_j^{(2)}$
for which $m_r^{(1)}=m_r^{(2)}$ for all $r$ is some nonempty subset
$I$ of $\{2,3, \ldots ,q\}$. There are $\prod_{j=2}^q 2^{i_j}$
choices for the sequence $m_j^{(1)}$ (as the value of $m_1^{(1)}$ is
determined by the value of the sum $\sum_{j=1}^q m_j^{(1)} a_j$).
The probability that each $m_j^{(1)}$ lies in $M_{a_j,i_j}$ is
$\frac{i_j^6}{2^{i_j}}$. For each fixed choice of $m_j^{(1)}$ and
for each nonempty subset $I$ as above, there are $\prod_{j>1, j
\not \in I} (2^{i_j}-1)$ possibilities to choose the numbers
$m_j^{(2)}, j>1, j \not \in I$, and the value of $m_1^{(2)}$ is
determined (and has to differ from $m_j^{(1)}$, which is another
reason the above is an upper estimate for $\Delta_1$ and not a
precise computation). The probability that 
$m_j^{(2)} \in M_{a_j, i_j}$ for all these values of $j$ is
$\frac{i_1^6}{2^{i_1}} \prod_{j>1, j \not \in I}
\frac{i_j^6}{2^{i_j}}$,
implying Claim 1.

Plugging the value of $\mu$ in Claim 1, we conclude that
\begin{equation}
\label{e32}
\Delta_1 \leq \mu^2 
\sum_{\emptyset \neq I \subset \{2,3,\ldots,q\}}
\prod_{j \in I} \frac{1}{i_j^6}
=\mu^2 ([\prod_{j=2}^q(1+\frac{1}{i_j^6})]-1) \leq
\mu^2 \frac{1}{\log^5 n}.
\end{equation}
The final inequality above follows from the fact that 
$i_j > 0.5 \log n$ for all $j$ and that $q=o(\log n)$.
\vspace{0.3cm}

\noindent
{\bf Claim 2:}\, 
For any $2 \leq \ell \leq q$,
$$
\Delta_{\ell} \leq \frac{1}{2^{i_1}} \prod_{j=1}^q i_j^6
\cdot \frac{i_{\ell}^6}{2^{i_{\ell}}}
\sum_{I \subset \{\ell+1,\ell+2,\ldots,q\}}
\prod_{r \in \{\ell+1, \ell+2, \ldots ,q\}-I} i_r^6.
$$

The proof is similar to that of Claim 1 with a few modifications.
For each choice of the sequence $m_j^{(1)}$ the
contribution to $\Delta_{\ell}$  arises from sequences $m_j^{(2)}$
for which $m_r^{(1)}=m_r^{(2)}$ for all $r<\ell$, $m_{\ell}^{(1)}
\neq m_{\ell}^{(2)}$, and there is a possibly empty
subset
$I$ of $\{\ell+1,\ell+2, \ldots ,q\}$ so that
$m_r^{(1)}=m_r^{(2)}$ for all $r \in I$ (and for no other $r$
in $\{\ell+1,\ell+2, \ldots ,q\}.$)
As in the proof of Claim 1 there are $\prod_{j=2}^q 2^{i_j}$
choices for the sequence $m_j^{(1)}$, and
the probability that each $m_j^{(1)}$ lies in $M_{a_j,i_j}$ is
$\frac{i_j^6}{2^{i_j}}$. For each fixed choice of $m_j^{(1)}$ and
for each subset $I$ as above, there are 
$\prod_{r \in \{\ell+1, \ell+2, \ldots ,q\}-I} (2^{i_r}-1)$ possible
choices for $m_r^{(2)}$, $r \in 
\{\ell+1, \ell+2, \ldots ,q\}-I$, and the probability that each
of those lies in the corresponding $M_{a_r,i_r}$ is
$\frac{i_r^6}{2^{i_r}}$. The product of these two terms is at most
$\prod_{r \in \{\ell+1, \ell+2, \ldots ,q\}-I} i_r^6.$
Finally, the value of $m_{\ell}^{(2)}$ is determined by the values
of all other $m_j^{(2)}$ and by the fact that 
$\sum_{j=1}^q m_j^{(2)} a_j =n$. (Note that this value has to lie
in $[2^{i_{\ell}}]$, otherwise we do not get any contribution here.
This is fine, as we are only upper bounding $\Delta_{\ell}$.)
Finally, the probability that $m_{\ell}^{(2)} \in
M_{a_{\ell},i_{\ell}}$ is $\frac{i_{\ell}^6}{2^{i_{\ell}}} $. This
completes the explanation for the estimate in Claim 2.

Plugging the value of $\mu$ we get, by Claim 2, that for any
$2 \leq \ell \leq q$
$$
\Delta_{\ell}
\leq \mu^2 \frac{2^{i_1}}{2^{i_{\ell}}} 
\frac{1}{i_1^6 i_2^6 \cdots i_{\ell-1}^6} 
\sum_{I \subset \{\ell+1,\ell+2,\ldots,q\}} \prod_{r \in I}
\frac{1}{i_r^6}
$$
$$
\leq \mu^2 \frac{2^{i_1}}{2^{i_{\ell}}} 
\frac{1}{i_1^6 i_2^6 \cdots i_{\ell-1}^6} \prod_{r=\ell+1}^q
(1+\frac{1}{i_r^6}) \leq 4 \mu^2 a_{\ell} \log n (\frac{2^6}{\log^6
n})^{\ell -1}.
$$
In the last inequality we used the fact that
$$
\frac{2^{i_1}}{2^{i_{\ell}}} \leq 2 a_{\ell} \log n~\mbox{ and } 
\prod_{r=\ell+1}^q(1+\frac{1}{i_r^6})<2.
$$
Since $a_{2}=O(1)$ we conclude that
$\Delta_2 \leq O(\frac{\mu^2 }{\log^5 n})$.
For any $\ell \geq 3$
we use the fact that $a_{\ell} < \ell^{(1+o(1)) \ell} \leq 
(\log n)^{\ell}$ to conclude that
$\Delta_{\ell} \leq O(\frac{\mu^2 }{\log^{(5-o(1)) \ell -7} n})
\leq O(\frac{\mu^2}{\log^7 n})$. 

Summing over all values of $\ell$ we conclude that
$\Delta=O(\frac{\mu^2}{\log^ 5 n})$. The assertion of Lemma
\ref{l33} thus follows from Lemma \ref{l34}. As mentioned after the
statement of this lemma, this also completes the proof of Theorem
\ref{t12}.
\end{proof}

\section{Concluding Remarks}

\begin{itemize}
\item
Since the proof of Theorem \ref{t12} 
is probabilistic and the probabilistic estimates
are strong enough it follows that for any finite collection of 
sequences $A_j$, each satisfying the assumptions of 
Theorem \ref{t12}, there is a sequence of multiplicities $M$ that is
good for each of them, where the number $n_0$ here depends on
all sequences $A_j$.
\item
The proof of Theorem \ref{t12} can be easily modified to work for 
any sequence $A$ that grows to infinity at
least as fast as $k^{\Omega(k)}$ and satisfies $gcd(A)=1$.
\item
Erd\H{o}s and Tur\'an \cite{ET} asked 
if for any asymptotic
basis of order $2$ of the positive integers (that is, a set $A$ of positive
integers so that each sufficiently large  integer has a
representation as a sum of two elements of $A$), there must be,
for any constant $t$,
integers that have more than $t$ such representations.
Theorem \ref{t11} shows that a natural analogous statement does not hold
for partition functions with restricted multiplicities.
\end{itemize}
\vspace{0.5cm}

\noindent
{\bf Acknowledgment}\,
Part of this work was carried out during a visit at the 
Ecole Polytechnique
in Palaiseau, France. I would like to thank 
Julia Wolf for her
hospitality during this visit. I would also like to thank
Mel Nathanson for communicating the Canfield-Wilf problem, and thank
him and Julia for helpful comments.

\end{document}